\documentclass[10pt]{amsart}
\usepackage[margin=1in,letterpaper,portrait]{geometry}
\usepackage{amssymb,amsmath}
\usepackage{verbatim} 

\usepackage[english]{babel} 
\usepackage{csquotes} 

\usepackage[]{xcolor} 
\usepackage[colorlinks=true, citecolor=blue, urlcolor=violet, breaklinks=true]{hyperref}

\newtheorem{Theorem}{Theorem}[section]
\newtheorem{Proposition}[Theorem]{Proposition}
\newtheorem{Lemma}[Theorem]{Lemma}
\newtheorem{Corollary}[Theorem]{Corollary}

\newtheorem{Question}[Theorem]{Question}
\theoremstyle{definition}

\newtheorem{Example}[Theorem]{Example}

\DeclareMathOperator{\Tor}{Tor}
\DeclareMathOperator{\sgn}{sgn}
\DeclareMathOperator{\Sym}{Sym} 

\DeclareMathOperator{\GL}{GL}

\newcommand{\kk}{\Bbbk}
\newcommand{\NN}{\mathbb{N}}
\newcommand{\ZZ}{\mathbb{Z}}

\newcommand{\SF}{\mathbb{S}} 

\newcommand{\mm}{\mathfrak{m}}
\newcommand{\Symm}{\mathfrak{S}} 

\newcommand{\ab}{\mathbf{a}} 
\newcommand{\bb}{\mathbf{b}}

\newcommand{\kellsize}{10}
\newlength{\kellsz} 
\newsavebox{\kell}
\sbox{\kell}{\begin{picture}(\kellsize,\kellsize)
	\put(0,0){\line(1,0){\kellsize}}
	\put(0,0){\line(0,1){\kellsize}}
	\put(\kellsize,0){\line(0,1){\kellsize}}
	\put(0,\kellsize){\line(1,0){\kellsize}}
	\end{picture}}
\newcommand\kellify[1]{\def\thearg{#1}\def\nothing{}%
	\ifx\thearg\nothing
	\vrule width0pt height\kellsz depth0pt\else
	\hbox to 0pt{\usebox{\kell} \hss}\fi%
	\vbox to \kellsz{
		\vss
		\hbox to \kellsz{\hss$#1$\hss}
		\vss}}
\newcommand\ktableau[1]{\vtop{\let\\\cr
		\baselineskip -16000pt \lineskiplimit 16000pt \lineskip 0pt
		\ialign{&\kellify{##}\cr#1\crcr}}}
\newcommand{\sellsize}{22}
\newlength{\sellsz} 
\newsavebox{\sell}
\sbox{\sell}{\begin{picture}(\sellsize,20)
	\put(0,0){\line(1,0){\sellsize}}
	\put(0,0){\line(0,1){\sellsize}}
	\put(\sellsize,0){\line(0,1){\sellsize}}
	\put(0,\sellsize){\line(1,0){\sellsize}}
	\end{picture}}
\newcommand\sellify[1]{\def\thearg{#1}\def\nothing{}%
	\ifx\thearg\nothing
	\vrule width0pt height\sellsz depth0pt\else
	\hbox to 0pt{\usebox{\sell} \hss}\fi%
	\vbox to \sellsz{
		\vss
		\hbox to \sellsz{\hss$#1$\hss}
		\vss}}
\newcommand\stableau[1]{\vtop{\let\\\cr
		\baselineskip -16000pt \lineskiplimit 16000pt \lineskip 0pt
		\ialign{&\sellify{##}\cr#1\crcr}}}
\begin{document} 

\title[Koszul homology of the second Veronese]{The Koszul homology algebra of the second Veronese is generated by the lowest strand}

\author{Aldo Conca}
\address{Dipartimento di Matematica, Universit\'a Degli Studi di Genova, 16146 Genova, Italy}
\email{conca@dima.unige.it}

\author{Lukas Katth\"an}
\address{School of Mathematics, University of Minnesota, Minneapolis, MN 55455, USA}
\email{katth001@umn.edu}

\author{Victor Reiner}
\address{School of Mathematics, University of Minnesota, Minneapolis, MN 55455, USA}
\email{reiner@math.umn.edu}

\dedicatory{To Craig Huneke on the occasion of his 65th birthday}
\keywords{}

\subjclass[2010]{Primary: 13D02; Secondary: 05E10, 20C15.}

\thanks{The second author was supported by the Deutsche Forschungsgemeinschaft (DFG), grant KA 4128/2-1.}

\begin{abstract}
We show that the Koszul homology algebra of the second Veronese subalgebra of a polynomial ring over a field of characteristic zero is generated, as an algebra,  by the homology classes corresponding to the syzygies of the lowest linear strand. 
\end{abstract}

\maketitle

\section{Introduction} 
Let $\kk$ be a field and $R$ be a standard graded $\kk$-algebra.
Let  $K(R)=\bigoplus K_i(R)$ be the Koszul complex with respect to a minimal system of generators of $R_1$ and let $H(R)=\bigoplus H_i(R)$ be its homology.
Indeed $H(R)$ is bigraded, 
\[
	H(R)=\bigoplus_{(i,j)} H_i(R)_j
\] 
where the first index $0\leq i\leq \dim_\kk R_1$ is the homological degree and the second index $j\in \NN$ is the internal degree coming from the graded structure of $R$.
In particular, $\dim_\kk H_i(R)_j$ is the $(i,j)$-th graded Betti number of $R$ over the symmetric algebra $\Sym_\kk(R_1)$ of $R_1$.
The  differential graded algebra (DGA) structure on $K(R)$ induces a graded-commutative $\kk$-algebra structure on  $H(R)$ that reflects important features of $R$.
For example, the Tate-Assmus theorem \cite[Thm. 2.3.11]{BH} asserts that $R$ is a complete intersection if and only if $H(R)$ is generated, as a $\kk$-algebra, by $H_1(R)$, and the Avramov-Golod theorem \cite[Thm. 3.4.5]{BH} asserts that $R$ is Gorenstein if and only if $H(R)$ is a Poincare duality algebra. 
 
Recall that $R$ is Koszul if $\kk$ has a linear resolution as an $R$-module.   It turns out that also the Koszul property of $R$ imposes strong  restrictions on the algebra structure of $H(R)$. 
It has been proved in \cite[Thm.3.1]{ACI1} and \cite[Thm.5.1]{ACI2} that if $R$ is Koszul then $H_i(R)_j=0$ for $j>2i$ and the multiplication map $\bigwedge^i H_1(R)_2 \to H_i(R)_{2i}$ is surjective. These results indicate that the  Koszul homology   $H(R)$ of a Koszul algebra $R$ tends to be generated by elements of  low degrees and  lead Avramov to ask the following question.  

\begin{Question} 
	Is the Koszul homology algebra $H(R)$ of a Koszul algebra $R$ generated, as a $\kk$-algebra, by the lowest linear strand, i.e. by $\bigoplus_{i} H_i(R)_{i+1}$?
\end{Question}

 In \cite[Thm.3.1]{BDGMS} the authors show that if $R$ is Koszul then the map  $\bigwedge^{i-2} H_1(R)_2 \otimes H_2(R)_3  \to H_i(R)_{2i-1}$ is surjective. On the other hand they show that the answer to   the  question is negative in general, see \cite[Thm. 3.15]{BDGMS}.  It remains open whether the answer is  positive for classical algebras such as Veronese subalgebras  and Segre products of polynomial rings.  The goal of this paper is to prove that the answer is indeed positive for  the second  Veronese subring of a polynomial ring in characteristic $0$.

Our approach will be to prove a somewhat stronger assertion, explained here, via combinatorial representation theory.   
Let  $S=\kk[x_1,\dots, x_n]=\bigoplus_i S_i$. The second Veronese subalgebra $S^{(2)}=\bigoplus_d  S_{2d}$ is a direct summand of $S$.
This implies that  the Koszul homology algebra $H(S^{(2)})$ of $S^{(2)}$ is a direct summand of the Koszul homology $H(\mm^2,S)$ associated to the second power $\mm^2$ of the maximal ideal $\mm=(x_1,\dots, x_n)$ of $S$, see for example \cite[Lemma 4.1]{BCR1} for the precise statements.
In \cite[Lemma 3.4, Remark 3.5]{BCR1} the authors describe a set of cycles  of $K(\mm^c,S)$ and provide some evidence that they generate (as an algebra) the algebra of Koszul cycles or, equivalently, the Koszul homology.
We will show that this is indeed the case for $H(\mm^2,S)$, see Theorem \ref{thm:mingen}. 
Then it will follow easily that  $H(S^{(2)})$ is generated by the elements in the lowest linear strand, see
Corollary~\ref{Veronese-algebra-corollary}.

\section{The set-up and statement of the results}

Let $\kk$ be a field of characteristic zero and let $S = \bigoplus_{d \geq 0} S_d \, \cong \, \kk[x_1,\ldots,x_n]$ be the polynomial ring.

The \emph{second Veronese subalgebra} of $S$ is 
\[
	S^{(2)} := \bigoplus_d  S_{2d}.
\]
Further, let $A := \kk[y_{ij}]$ be the polynomial ring in the entries of a generic symmetric $n \times n$ matrix.
The map $y_{ij} \mapsto x_i x_j$ turns $S$ and $S^{(2)}$ into $A$-algebras.
Thus, one can compute the Koszul homology of $S^{(2)}$ as $\Tor_*^A(S^{(2)}, \kk)$.
It turns out to be more convenient to consider $\Tor_*^A(S, \kk)$ instead. Since the former is a direct summand of the latter, we can still extract information about the Koszul homology of $S^{(2)}$ from this.

An alternative point of view is the following. Let $\mm=(x_1,\ldots,x_n)$ denote the graded maximal ideal in $S$.
Its square $\mm^2$ is generated by the set of all quadratic monomials in $S$.
The Koszul complex 
$$K=K(\mm^2)=K(\mm^2,S)$$  for this generating set
is a DGA-structure on $S \otimes \bigwedge^* S_2$, 
whose homology $H_*(\mm^2)$,
as a bigraded vector space, is isomorphic to $\Tor_*^A(S,\kk)$.

Let $V$ be the vector space generated by the $x_i$. Then $S$ can be considered as the symmetric algebra $\Sym_\kk(V)$ and the general linear group $\GL(V)$ acts on it, preserving the usual grading.
Thus $H_*(\mm^2)$ is also representation of $\GL(V)$, that is, a $\kk[\GL(V)]$-module.
Our first result is to give a set of $\kk[\GL(V)]$-module generators for  $H_*(\mm^2)$.

For $1 \leq i \leq n$, consider the element
\[ 
Z_{i-1} := \sum_{j = 1}^{i} (-1)^{j-1} 
x_j \otimes x_1x_i \wedge x_2x_{i} \wedge \dotsc \wedge \widehat{x_jx_i} \wedge \dotsc \wedge x_{i-1}x_{i}  \wedge x_i^2
\]
Each $Z_i$ is a Koszul cycle in $S_1 \otimes \bigwedge^i S_2$,
representing a (non-zero) homology class in $H_i(\mm^2)_{2i+1}$.

\begin{Example}
	Explicitly, one has
	\begin{align*}
	Z_0 &= x_1 \otimes 1, \\
	Z_1 &= x_1 \otimes x_2^2 - x_2 \otimes x_1x_2, \\
	Z_2 &= x_1 \otimes x_2x_3 \wedge x_3^2 - x_2 \otimes x_1x_3 \wedge x_3^2 + x_3 \otimes x_1x_3 \wedge x_2x_3, \\
	\vdots & \\
	\end{align*}
\end{Example}

\begin{Theorem}
	\label{thm:main1}
	As a $\kk[\GL(V)]$-module, the Koszul homology $H_{*}(\mm^2)$ is generated by all squarefree monomials in $Z_0, \dotsc, Z_{n-1}$.
\end{Theorem}

Next, we construct a minimal generating set for $H_{*}(\mm^2)$ as  an algebra over  $\kk$.
For each choice of an integer $t \geq 0$ and two sequences
$\ab=(a_1,\ldots,a_t,a_{t+1}), \bb=(b_1,\ldots,b_t)$ of linear forms in $S_1$, consider the element
\begin{equation}
\label{BCR-cycle}
z_{\ab,\bb}:=
\sum_{\sigma \in \Symm_{t+1}} 
\sgn(\sigma) \,\, a_{\sigma(t+1)} \otimes (b_1 a_{\sigma(1)} \wedge \cdots \wedge
b_t a_{\sigma(t)}) 
\end{equation}
of $S_1 \otimes \bigwedge^t S_2$.
Here $\Symm_{t+1}$ denotes the symmetric group on $t+1$ letters, permuting the subscripts in $(a_1,\ldots,a_{t+1})$, and $\sgn(\sigma)$ is the $\{ \pm 1\}$ sign of the permutation $\sigma$.

\begin{Theorem}\label{thm:mingen}
 The Koszul homology $H_*(\mm^2)$ is minimally generated as an algebra over $\kk$  by the homology classes of   $z_{\ab,\bb}$ for sequences of variables $\ab = (x_{i_1}, \dotsc, x_{i_{t+1}})$ and $\bb = (x_{j_1}, \dotsc, x_{j_{t}})$ such that
	\begin{enumerate}
		\item[(i)] $i_1 < i_2 < \dotsc < i_{t+1}$,
		\item[(ii)]  $j_1 \leq j_2 \leq \dotsc \leq j_{t}$, and
		\item [(iii)] $i_1 \leq j_1$.
	\end{enumerate}
\end{Theorem}
\noindent
Equivalently,  the  above conditions on $\ab, \bb$ assert that the following is a semistandard (column-strict) tableau:
$$
\stableau{ 
{i_1} & {j_1} & {j_2} & {\cdots} & {i_t}\\
{i_2} \\
{\vdots} \\
{i_t} \\
{i_{t+1}}
}
$$

The cycles $z_{\ab,\bb}$ were introduced by Bruns, Conca and R\"omer \cite[Lemma 3.4, Remark 3.5]{BCR1} in a more general context. There, it was suggested that they give generating sets for the Koszul homology of all Veronese subalgebras of the polynomial ring.
In \cite[Prop. 5.5]{BCR2}, they provide some evidence that these
elements generate, by showing that a subset of them generate
the cycles of $K(\mm^c)$ as an algebra up through homological dimension $2$ for each $c$.

\begin{Corollary}
\label{Veronese-algebra-corollary}
The Koszul homology of the second Veronese algebra is generated by  its lowest strand.
\end{Corollary}
\begin{proof}
	The $A$-module $S$ splits into two direct summands, $M_1, M_2$, where $M_1 \cong S^{(2)}$ contains all elements of even degree and $M_2$ contains all elements of odd degree. 
	Since all maps in a minimal free resolution of $S$ over $A$ are homogeneous, this carries over to $\Tor_*^A(S, \kk)$. 
	Thus $\Tor_*^A(S^{(2)}, \kk)$ is itself the second Veronese subalgebra of $\Tor_*^A(S, \kk)$ (with respect to the internal degrees).
	
	Together with Theorem \ref{thm:mingen}, this implies that every element of $\Tor_*^A(S^{(2)}, \kk)$ can be written as a sum of products of an even number of generators of the form $z_{\ab,\bb}$. So $\Tor_*^A(S^{(2)}, \kk)$ is generated by the quadratic monomials in the $z_{\ab,\bb}$.
	But those lie in the lowest strand of it, and thus the claim follows.
\end{proof}

Finally, we present an example showing how our results can fail when $\kk$ has positive 
characteristic.

\begin{Example}
	In characteristic $2$, the $z_{\ab,\bb}$ do not generate $H_*(\mm^2)$.
	Consider the element 
	\begin{align*}
		z := x_1 \otimes x_2x_3 \wedge x_4x_5 + x_3 \otimes x_4x_5 \wedge x_1x_2 + x_5 \otimes x_1x_2 \wedge x_3x_4 + x_2 \otimes x_3x_4 \wedge x_1x_5 + x_4 \otimes x_1x_5 \wedge x_2x_3  
	\end{align*}
	of $\Sym^1(V) \otimes \bigwedge^2 \Sym^2(V)$.
	The sum of the coefficients of the terms of $z$ is $1 \bmod{2}$. On the other hand, every $z_{\ab,\bb}$ for $t = 1$ has six terms and thus coefficient sum $0$. Hence $z$ is not a linear combination of $z_{\ab,\bb}$ when $\kk$ has characteristic $2$.
	
	A combinatorial interpretation of this example is as follows.  It is well-known (see, e.g., \cite[Prop. 3.2]{ReinerRoberts}) that, within  $H_i(\mm^2)$ (or $\Tor^A_i(S,\kk)$), 
	the homogeneous component with
	square-free multidegree $x_1 x_2 \cdots x_n$ is
	isomorphic as a $\kk$-vector space to the (reduced) simplicial homology $\tilde{H}_{i-1}(\Delta_n; \kk)$,
	where $\Delta_n$ is the {\it matching complex} on $n$ vertices.  That is, $\Delta_n$ is the simplicial complex whose vertex set is indexed by all pairs $\{i,j\}$ with $1 \leq i < j \leq n$, and whose simplices are collections $\{ \{i_m,j_m\}\}_{m=1,2,\ldots,t}$ which are pairwise disjoint, that is, they form the edges of a partial matching inside the complete graph $K_n$.  It turns out that $\Delta_5$ is the Petersen graph (see \cite[Fig. 4(b)]{ReinerRoberts}, and that under the above isomorphism, the elements $z_{\ab,\bb}$ correspond to cycles of length $6$ in the Petersen graph, while $z$ corresponds to a cycle of length $5$.
	By the parity argument above, the $5$-cycle cannot be obtained as a linear combination of the $6$-cycles.
\end{Example}

The following questions remain open.

\begin{Question}
 In positive odd characteristic, do the cycles  $z_{\ab,\bb}$   span the lowest strand of the Koszul homology $H_*(\mm^2)$?
\end{Question}

\begin{Question}
 In positive characteristic (or over integers), are the Koszul homologies 
$H_*(\mm^2)$ or $H_*(S^{(2)})$  generated as algebras over
$\kk$ by their lowest strands? 
\end{Question}

\section{Generating the homology as a \texorpdfstring{$\kk[\GL(V)]$}{k[GL(V)]}-module}
\label{sec:generation-for-c=2}

The Koszul homology $H_*(\mm^2) \cong \Tor_*^A(S,\kk)$ is not only bigraded as a $\kk$-vector space,
but also a bigraded representation for $\GL(V) \cong \GL_n(\kk)$.
There is a very simple, explicit description of the
multiplicities of its $\GL(V)$-decomposition given by 
results of J\'ozefiak-Pragacz-Weyman \cite{JPW}
and Reiner-Roberts \cite{ReinerRoberts}.

When considered as a function from $\kk$-vector spaces $V$
to (bigraded) $\GL(V)$-representations, the map $V \longmapsto H_*(\mm^2)$
is also a {\it polynomial functor}, in the sense of Macdonald \cite[Chap. I, App. A]{Macdonald}.
Therefore each graded component decomposes into irreducible polynomial
$\GL(V)$-representations, which are the {\it Schur functors} $\SF^\lambda(V)$
indexed by partitions 
$
\lambda=(\lambda_1 \geq \cdots \geq \lambda_n)
$
whose length (=number of nonzero parts $\lambda_i$) is at most $n=\dim_\kk V$,
that is, $\lambda_n \geq 0$; the weight $|\lambda|:=\sum_i \lambda_i$ gives
the homogeneity of $\SF^\lambda(V)$ as a polynomial representation. 

It turns out that the only $\GL(V)$-irreducibles $\SF^\lambda(V)$
which appear in $H_*(\mm^2)$  correspond to {\it self-conjugate}
partitions $\lambda$, that is, those whose Ferrers diagram 
is symmetric under reflection across the diagonal.  Each such $\lambda$
can be expressed 
$
\lambda=
\left( \begin{smallmatrix} 
\mu \\ 
\mu
\end{smallmatrix} \right)
=
\left( \begin{smallmatrix} 
\mu_1 & \cdots & \mu_s \\ 
\mu_1 & \cdots & \mu_s
\end{smallmatrix} \right)
$
in the {\it Frobenius notation} for partitions \cite[\S I.1, page 3]{Macdonald}, where
$s:=\max\{i: \lambda_i \geq i\}$ is the {\it Durfee square size}
of $\lambda$, and $\mu_i=\lambda_i-i$ for $i=1,2,\ldots,s$.
The following statement combines  \cite[Thm. 3.19 at $r=1$]{JPW}
and \cite[Thm. 1.2]{ReinerRoberts}.

\begin{Theorem}
\label{thm:second-Veronese-Koszul-description}
As a bigraded $\GL(V)$-representation 
$$
H_*(\mm^2) \cong \Tor_*^A(S,\kk) \cong \bigoplus_{\text{self-conjugate }\lambda} \SF^\lambda(V)
$$
where the direct sum runs over all self-conjugate partitions
$\lambda=\left( \begin{smallmatrix} \mu \\ \mu \end{smallmatrix} \right)$ with $\ell(\lambda) \leq n$,
and with the summand $\SF^\lambda(V)$ occurring uniquely in the bidegree
$
H_{|\mu|}(\mm^2)_{|\lambda|} \cong \Tor_{|\mu|}^A(S,\kk)_{|\lambda|}.
$
\end{Theorem}

\begin{Example}
The self-conjugate 
$
\lambda=(7,7,6,3,3,3,2)
=\left( \begin{smallmatrix} 
6 &5 &3  \\ 
6 &5 &3
\end{smallmatrix} \right)
$
has Durfee square of size $s=3$:
$$
\lambda 
= \ktableau{ 
{\times} & {} & {} & {} & {} & {} & {} \\
{} & {\times} & {} & {} & {} & {} & {} \\
{} & {} & {\times} & {} & {} & {} \\
{} & {} & {} \\
{} & {} & {} \\
{} & {} & {} \\
{} & {} }
$$
Here $\mu=(6,5,3)$ gives the sizes of the rows strictly above the diagonal
cells, marked here with an ``$\times$''.
Since $|\mu|=14$, and $|\lambda|=31$,  the unique copy of $\SF^\lambda(V)$ appears in
bidegree $H_{14}(\mm^2)_{31}$, assuming $n=\dim_\kk V \geq 7$.
\end{Example}

Our method for showing that  the elements $z_{\ab,\bb}$ generate $H_*(\mm^2)$ is
motivated by a proof of Euler's result asserting that
the number of self-conjugate partitions of $N$ is the same as the 
number of partitions of $N$ that use only 
only odd, distinct part sizes:  there is a bijection sending
the self-conjugate partition
$\lambda=\left( \begin{smallmatrix} 
\mu\\ 
\mu
\end{smallmatrix} \right)$ to the partition consisting of the (odd) sizes $(2\mu_1+1,\ldots,2\mu_s+1)$
of its $s$ (distinct) self-conjugate hooks on the main diagonal 
$$
\left\{ 
\left( \begin{smallmatrix} 
\mu_1 \\ 
\mu_1
\end{smallmatrix} \right)
,\ldots, 
\left( \begin{smallmatrix} 
\mu_s \\ 
\mu_s
\end{smallmatrix} \right)
\right\}
= 
\left\{ 
(\mu_1+1,1^{\mu_1})
,\ldots, 
(\mu_s+1,1^{\mu_s})
\right\}.
$$

\begin{Example}
\label{ex:Euler-bijection}
The self-conjugate $\lambda$ above
is sent to $(11,9,5)$ having distinct odd parts:
$$
\ktableau{ 
{\times} & {} & {} & {} & {} & {} & {} \\
{} & {\times} & {} & {} & {} & {} & {} \\
{} & {} & {\times} & {} & {} & {} \\
{} & {} & {} \\
{} & {} & {} \\
{} & {} & {} \\
{} & {} }
=
\ktableau{ 
{\times} & {} & {} & {} & {} & {} & {} \\
{} & \\
{} & \\
{} & \\
{} & \\
{} & \\
{} & }
\quad \sqcup \quad
\ktableau{ 
{\times} & {} & {} & {} & {} & {} \\
{} & \\
{} & \\
{} & \\
{} & \\
{} & }
\quad \sqcup \quad
\ktableau{ 
{\times} & {} & {} & {}  \\
{} & \\
{} & \\
{} & }
$$
\end{Example}

\noindent
The role of the hooks will be played by the elements $Z_i$ defined above.
The goal of this section is the following more precise version of Theorem \ref{thm:main1}:
\begin{Theorem}
\label{thm:main}
For each self-conjugate 
$\lambda=\left( \begin{smallmatrix} \mu\\ \mu\end{smallmatrix} \right)
=\left( \begin{smallmatrix} \mu_1 & \cdots & \mu_s \\ \mu_1 &\cdots & \mu_s \end{smallmatrix} \right)$ 
with $\ell(\lambda) \leq n$,
the monomial $Z_{\mu_1} \cdots Z_{\mu_s}$ generates the 
$\SF^\lambda(V)$-isotypic component of $H_{|\mu|}(\mm^2)_{|\lambda|}$
as an (irreducible) $\kk[\GL(V)]$-module.
\end{Theorem}

\noindent
Since irreducibility shows that any nonzero vector in $\SF^\lambda(V)$ generates it as a
$\kk[\GL(V)]$-representation, and since Theorem~\ref{thm:second-Veronese-Koszul-description}
shows $\SF^\lambda(V)$ occurs in $H_*(\mm^2)$ with multiplicity one, these two lemmas prove the theorem:

\begin{Lemma}
\label{lem:products-are-nonzero}
	All squarefree products of $\{Z_0,Z_1,\ldots,Z_{n-1}\}$ are nonzero in the Koszul homology $H_*(\mm^2)$.
\end{Lemma}

\begin{Lemma}
\label{lem:products-are-isotypic}
	The product $Z_{\mu_1} \cdots Z_{\mu_s}$ in Theorem~\ref{thm:main}
	lies in the $\SF^\lambda(V)$-isotypic component of $H_*(\mm^2)$.
\end{Lemma}

\begin{proof}[Proof of Lemma~\ref{lem:products-are-nonzero}]
Since every squarefree product is a factor of the full product $Z_0Z_1Z_2\dotsm Z_{n-1}$,
it suffices to show that the latter is nonzero in  $H_*(\mm^2)$. First note that for each $i$,
	\[
	Z_{i} = (-1)^{i} x_{i+1} \otimes x_1x_{i+1} \wedge x_2x_{i+1}  \wedge \dotsc \wedge x_{i} x_{i+1} + \text{[terms with $x_{i+1}^2$]}
	\]
Hence each $Z_i$ has exactly one term which contains no square.
	Therefore $Z_0 Z_1 \dotsm Z_{n-1}$ contains the term
	\[ T := \pm x_1 x_2 \dotsm x_n \otimes \bigwedge_{i=1}^n \left(  \bigwedge_{j=1}^{i-1} x_j x_i \right) \]
	 and all other terms contain a square.
	Thus $T$ cannot cancel against any other term, and so $Z_0 Z_1 \dotsm Z_{n-1}$ is non-zero as an element of the Koszul complex.  To see that $Z_0Z_1Z_2\dotsm Z_{n-1}$ is also non-zero in homology, we claim that $T$ does not occur in the boundary of any element of the Koszul complex. Indeed, a basis element whose boundary involves $T$ would need to be a scalar multiple of 
	\[ \frac{x_1 x_2 \dotsm x_n}{x_{j_1}x_{j_2}} \otimes x_{j_1}x_{j_2} \wedge \bigwedge_{i=1}^n \left( \bigwedge_{j=1}^{i} x_j x_i \right)\]
with $1\leq j_1 < j_2\leq n$, but  would vanish since
	 $\bigwedge_{i=1}^n \bigwedge_{j=1}^{i} x_j x_i$ already contains every such product 
	 $x_{j_1}x_{j_2}$.
\end{proof}

We will need to consider the Koszul complex, and its homology, for {\it varying} spaces $V$. Let 
$K_\bullet(V)$ denote the Koszul complex of the second Veronese subalgebra of the symmetric algebra on $V$, i.e.
\[ 
K_\bullet(V) = K(\mm^2, \Sym(V)), \quad \text{ where } \mm:=\bigoplus_{d \geq 1}\Sym^d(V). 
\]
This is a {\it polynomial functor} of $V$ in the sense of \cite[Ch. I, App. A]{Macdonald}.
Further, we define $H_*(V) := H_*(K_\bullet(V))$, which is again a polynomial functor of $V$,
because sub- and quotient-functors of polynomial functors are again polynomial.
For a $\GL(V)$-representation $W$, we write $I_\lambda W$ for the $\SF^\lambda(V)$-isotypic component of $W$.  We will need a well-known fact about the interaction of polynomial functors with isotypic components.

\begin{Proposition}
\label{prop:functoriality-of-isotypic-components}
For any polynomial functor $V \longmapsto F(V)$, and any linear map $U \overset{f}{\longrightarrow} V$,
the induced map $F(U) \overset{F(f)}{\longrightarrow} F(V)$ preserves isotypic components, that is,
$
F(f)(I_\lambda F(U)) \subset I_\lambda F(V).
$
\end{Proposition}
\begin{proof}
One always has a direct sum decomposition of polynomial functors
$F=\oplus_\lambda F_\lambda$ with the property that $F_\lambda(V)=I_\lambda(F(V))$; see
\cite[Ch. I, App. A, (5.5)]{Macdonald}. The proposition then follows.
\end{proof}

A key step in proving Lemma~\ref{lem:products-are-isotypic} is a certain combinatorial lemma.
For any self-conjugate partition
$\lambda=\left( \begin{smallmatrix} \mu \\ \mu \end{smallmatrix} \right)
=\left( \begin{smallmatrix} \mu_1 & \cdots & \mu_s \\  \mu_1 & \cdots & \mu_s \end{smallmatrix} \right),
$
define two smaller self-conjugate partitions
$
\nu=\left( \begin{smallmatrix} \mu_1 \\ \mu_1\end{smallmatrix} \right)
$
and
$
\hat{\lambda}:=\left( \begin{smallmatrix} \mu_2 & \cdots & \mu_s \\  \mu_2 & \cdots & \mu_s \end{smallmatrix} \right).
$
E.g., in Example~\ref{ex:Euler-bijection}, one has 
$$
\lambda=\left( \begin{smallmatrix} 
6 &5 &3  \\ 
6 &5 &3
\end{smallmatrix} \right)
=\ktableau{ 
{\times} & {} & {} & {} & {} & {} & {} \\
{} & {\times} & {} & {} & {} & {} & {} \\
{} & {} & {\times} & {} & {} & {} \\
{} & {} & {} \\
{} & {} & {} \\
{} & {} & {} \\
{} & {} }
\qquad
\nu=\left( \begin{smallmatrix} 
6  \\ 
6
\end{smallmatrix} \right) 
=\ktableau{ 
{\times} & {} & {} & {} & {} & {} & {} \\
{} & \\
{} & \\
{} & \\
{} & \\
{} & \\
{} & }
\quad
\hat{\lambda}=\left( \begin{smallmatrix} 
5 &3  \\ 
5 &3
\end{smallmatrix} \right)
=\ktableau{ 
 {\times} & {} & {} & {} & {} & {} \\
{} & {\times} & {} & {} & {} \\
{} & {} \\
{} & {} \\
{} & {}\\
{} } 
$$

\begin{Lemma}
\label{lem:LR}
With the above notation,  if $\dim_\kk V=\mu_1+1$, then the tensor product 
$\SF^{\nu}(V) \otimes \SF^{\hat{\lambda}}(V)$ contains
the irreducible $\SF^\lambda(V)$, but contains no other irreducibles $\SF^{\rho}(V)$ for
self-conjugate partitions $\rho \neq \lambda$.
\end{Lemma}
\begin{proof}
The {\it Littlewood-Richardson rule} \cite[\S I.9]{Macdonald} asserts that
$$
\SF^{\hat{\lambda}}(V) \otimes \SF^{\nu}(V)
= \bigoplus_\rho c_{\nu, \hat{\lambda}}^\rho \SF^\rho(V)
$$
where the sum is over partitions $\rho$ whose Ferrers diagram contains $\hat{\lambda}$,
with $c_{ \nu, \hat{\lambda}}^\rho$ counting these objects:
\begin{itemize}
\item {\it column-strict tableaux} of the {\it skew shape}
$\rho/\hat{\lambda}$, 
\item
filled with $\mu_1+1$ ones and exactly one occurrence of each of $2,3,\ldots, \mu_1+1$,
\item
containing a subsequence $1,2,3,\ldots,\mu_1,\mu_1+1$ read in a weakly southwesterly
fashion. 
\end{itemize} 
This means that this subsequence of $\mu_1+1$ letters will occupy a vertical strip (no two in the same row),
while the $\mu_1+1$ ones will occupy a horizontal strip (no two  in the same column).

  However, because $\dim_\kk V = \mu_1+1$, the summand $\SF^\rho(V)$ vanishes unless $\ell(\rho) \leq \mu_1+1$.  This forces the aforementioned vertical strip to occupy the top $\mu_1+1$ rows of $\rho$.
Since we are only concerned with {\it self-conjugate} $\rho$ having $\SF^\rho(V)$ in the expansion, 
one may also assume that $\rho_1 \leq \mu_1+1$, forcing the aforementioned horizontal strip to occupy the first $\mu_1+1$ columns of $\rho$.  Together, these two conditions force $\rho=\lambda$.
\end{proof}

\begin{proof}[Proof of Lemma~\ref{lem:products-are-isotypic}]	
We will show that the product $Z_{\mu_1} \dotsm Z_{\mu_{s-1}} Z_{\mu_s}$ lies in  $I_\lambda H_*(V)$ by induction on $s$.

A key fact to note, both for the base case $s=1$, and for the inductive step, is that
Theorem~\ref{thm:second-Veronese-Koszul-description} implies 
$H_{\mu_1}(V)_{2\mu_1+1} \cong \SF^{\nu}(V)$ where 
$\nu=\left( \begin{smallmatrix} \mu_1 \\ \mu_1\end{smallmatrix} \right).$
Since $Z_{\mu_1}$ lies in $H_{\mu_1}(V)_{2\mu_1+1}$, this take care of the base case.

In the inductive step, note that each of $Z_{\mu_2}, \ldots, Z_{\mu_{s-1}}, Z_{\mu_s}$
involves only variables $x_1,x_2,\ldots,x_{\mu_2+1}$, and therefore their
product $Z_{\mu_2}  \cdots Z_{\mu_{s-1}} Z_{\mu_s}$ may be considered 
an element of $H_*(U)$ where $U$ is the subspace of $V$ with 
$\kk$-basis $x_1,x_2,\ldots,x_{\mu_2+1}$.  By induction, this product 
lies in $I_{\hat{\lambda}} H_*(U)$.  Proposition~\ref{prop:functoriality-of-isotypic-components} applied to the 
inclusion $U \hookrightarrow V$ show that this same product, considered as an element of $H_*(V)$,
lies in $I_{\hat{\lambda}} H_*(V)$.  

As observed above, $Z_{\mu_1}$ lies in $I_\nu H_*(V)$,
so the product $Z_{\mu_1} \cdot \left( Z_{\mu_2} \dotsm Z_{\mu_{s-1}} Z_{\mu_s} \right)$
lies in the image of the multiplication map  
$
I_\nu H_*(V) \otimes I_{\hat{\lambda}} H_*(V) \rightarrow H_*(V).
$
Since Theorem~\ref{thm:second-Veronese-Koszul-description} implies 
$H_*(V)$ only contains irreducibles $\SF^\rho(V)$ with self-conjugate $\rho$,
Proposition~\ref{lem:LR} shows the image of this multiplication lies in $I_\lambda H_*(V)$.
\end{proof}

This completes the proof of Theorem~\ref{thm:main}.

\section{A minimal set of generators}
In this section, we prove Theorem \ref{thm:mingen}.
We will use the following technical lemma:

\begin{Lemma}[Garnir-type relations]\label{lem:garnir}
	Let $t \in \NN$. Let $\ab = (a_1, \dotsc, a_{t+2})$ and $\bb = (b_2, \dotsc, b_t)$ be sequences of variables, that is $a_i, b_j \in \{x_1,\ldots,x_n\}$, with repetitions allowed.
	For $j = 1,\dotsc, t+2$ set
	\begin{align*}
	\ab^{(j)} &:= (a_1, \dotsc, \widehat{a}_j, \dotsc, a_{t+2}), \,\, \text{and} \\
	\bb^{(j)} &:= (a_j, b_2, \dotsc, b_t). &
	\end{align*}
	Then 
	\[ \sum_{j = 1}^{t+2} (-1)^j z_{\ab^{(j)}, \bb^{(j)}} = 0 \]
\end{Lemma}
\begin{proof}
	We compute:
	\begin{align*}
	\sum_{j = 1}^{t+2} (-1)^j z_{\ab^{(j)}, \bb^{(j)}} 
	&= \sum_{j = 1}^{t+2} (-1)^j \sum_{\sigma\in\Symm_{t+1}} \sgn(\sigma) \,\, a^{(j)}_{\sigma(t+1)} \otimes a_ja^{(j)}_{\sigma(1)} \wedge b_2a^{(j)}_{\sigma(2)} \wedge \dotsm \wedge b_{t}a^{(j)}_{\sigma(t)} \\
	&=\sum_{(\tau, j)} (-1)^{j} \sgn(\tau) \,\, a_{\tau(t+1)} \otimes a_ja_{\tau(1)} \wedge b_2a_{\tau(2)} \wedge \dotsm \wedge b_{t}a_{\tau(t)}
	\end{align*}
	where the second sum runs over all injective maps $\tau: [t+1] \to [t+2]$ and $j \in [t+2]$ is the element which is not in the image of $\tau$.  Here $\sgn(\tau)$ is $+1,-1$ depending upon whether the number of inversion pairs $1 \leq i < j \leq t+1$ with $\tau(i) > \tau(j)$ is even or odd.
	
	We now show that the above sum vanishes via a sign-reversing involution on its terms. For each $\tau$ as above, define a new injection $\tilde{\tau}:[t+1] \to [t+2]$ that first applies $\tau$ and then swaps the values $\tau(1)$ and $j$, that is,
$$
\tilde{\tau}(i) = 
\begin{cases}
\tau(i) & \text{ if } i \neq 1,\\
j& \text{ if }i=1.
\end{cases}
$$
	Now, the pairs $(\tau,j)$ and $(\tilde{\tau},\tau(1))$ give rise to the same term in the sum. One can check that the signs of the two terms, namely $(-1)^{j} \sgn(\tau)$ for the former, versus $( -1)^{\tau(1)}\sgn(\tilde{\tau})$ for the latter,
	are opposite.
\end{proof}

\begin{proof}[Proof of Theorem \ref{thm:mingen}]
	It follows from Theorem \ref{thm:main} that $H_*(\mm^2)$ is minimally generated by any basis of the $\GL(V)$-re\-pre\-sentation generated by the $Z_t$.  
	We know that the $\GL(V)$-representation generated by $Z_t$ is the hook-shaped irreducible $\SF^{\nu}$ where $\nu=(t+1,1^t)$.
	The dimension of this representation is the number of semi-standard Young tableaux of that shape, and as discussed immediately after the statement of Theorem \ref{thm:mingen},  there is an obvious bijection between our proposed basis and the set of those tableaux.  Thus it suffices to show that our proposed basis is a spanning set.

	Fix the value of $t$. 
	Note that $z_{\ab,\bb}$ is linear in the entries of $\ab$ and $\bb$, and hence for any $g \in \GL(V)$, the element $g.z_{\ab,\bb}$ can be written as a sum of elements of the form $z_{\ab',\bb'}$. 
	Moreover, we only need to consider the case where $\ab$ and $\bb$ are sequences of variables, again by linearity.
	Further, $z_{\ab,\bb}$ is antisymmetric in the entries of $\ab$ and symmetric in the entries of $\bb$, hence we may assume that $(\ab,\bb)$ has $\ab$ strictly increasing and $\bb$ nondecreasing, that is, satisfying 
conditions (i), (ii) of the theorem.  
	
To show that all such $z_{\ab,\bb}$ lie in the span of those that also satisfy condition (iii), we will induct on
the quantity $\omega(\ab,\bb):=\#\{j: a_1 > b_j\}$.  For the base case, note that 
$(\ab,\bb)$ satisfies (3) if and only if $\omega(\ab,\bb) = 0$.
	If $\omega(\ab,\bb) > 0$, then let $\bar{\ab} := (b_1, a_1,\dotsc, a_{t+1})$ and $\bar{\bb} := (b_2, \dotsc, b_t)$.
	We apply Lemma \ref{lem:garnir} to this $\bar{\ab}$ and $\bar{\bb}$. It is easy to see that for each $j > 1$ one has 
	$$\omega(\bar{\ab}^{(j)}, \bar{\bb}^{(j)}) < \omega(\bar{\ab}^{(1)}, \bar{\bb}^{(1)})) = \omega(\ab,\bb),$$ 
	and so we can rewrite $z_{\ab,\bb}$ as a linear combination of elements $z_{\bar{\ab}^{(j)}, \bar{\bb}^{(j)}}$
	 to which induction applies.
\end{proof}

\end{document}